\documentclass[11pt,letterpaper]{amsart}
\usepackage[T1]{fontenc}
\usepackage[utf8]{inputenc}
\usepackage{lmodern}
\usepackage{amsmath,amssymb,amsthm}
\usepackage[letterpaper,margin=1in]{geometry}
\usepackage[expansion=false]{microtype}
\usepackage[hidelinks]{hyperref}
\theoremstyle{plain}
\newtheorem{theorem}{Theorem}[section]
\newtheorem{proposition}[theorem]{Proposition}
\newtheorem{lemma}[theorem]{Lemma}
\newtheorem{corollary}[theorem]{Corollary}
\theoremstyle{definition}
\newtheorem{definition}[theorem]{Definition}
\newtheorem{remark}[theorem]{Remark}

\newcommand{\Fp}{\mathbb{F}_p}
\newcommand{\Fpbar}{\overline{\mathbb{F}}_p}
\newcommand{\AG}{\mathrm{AG}}
\newcommand{\PG}{\mathrm{PG}}
\newcommand{\PP}{\mathbb{P}}
\newcommand{\AAA}{\mathbb{A}}
\newcommand{\NS}{\mathcal{N}}
\newcommand{\BB}{\mathcal{B}}
\newcommand{\Dick}{\mathsf{D}}
\DeclareMathOperator{\spec}{sp}
\DeclareMathOperator{\Val}{Val}

\title[The cylinder conjecture]{The weak cylinder conjecture for polynomial graphs of bounded degree}
\author{Daniele Bartoli}
\address{Dipartimento di Matematica e Informatica,
Universit\`a degli Studi di Perugia,
Via Vanvitelli 1, 06123 Perugia, Italy}
\email{daniele.bartoli@unipg.it}

\author{Alessandro Giannoni}
\address{Dipartimento di Matematica e Applicazioni ``Renato Caccioppoli'',
Universit\`a degli Studi di Napoli Federico II,
Complesso Universitario di Monte Sant'Angelo, Via Cintia,
80126 Napoli, Italy}
\email{alessandro.giannoni@unina.it}
\date{}
\subjclass[2020]{Primary 51E20; Secondary 11T06, 14G15}
\keywords{Cylinder conjecture, finite geometry, directions, exceptional polynomials, polynomial decomposition}
\hypersetup{pdftitle={The weak cylinder conjecture for polynomial graphs of bounded degree},pdfauthor={Daniele Bartoli, Alessandro Giannoni}}

\begin{document}

\begin{abstract}
Ball's weak cylinder conjecture asserts that a set of $p^2$ points in
$\AG(3,p)$ with at least $p$ non-determined directions is a union of
$p$ parallel lines. We prove the conjecture for sets admitting a
polynomial graph representation of total degree at most $D$, provided
that $p>D^{9/2}$. More precisely, a non-cylindrical graph of degree
$D\ge3$ satisfies
$|\mathcal N_S|\le p+1-\lceil p/(D-1)\rceil$ under the same hypothesis.
This bound is attained when $p\equiv1\pmod{D-1}$.
The proof relates non-determined directions to permutation polynomials
on parallel affine lines. A rigidity lemma for moving line restrictions,
together with a reconstruction of their compositional components,
bounds the number of contributing directions uniformly in $p$.
Directions with affine restrictions are treated separately; all other
contributions are bounded using geometric indecomposability and the
Weil bound for plane curves.
\end{abstract}

\maketitle

\section{Introduction}\label{sec:intro}

Let $p$ be a prime. A direction in $\AG(3,p)$ is a point of the
plane at infinity $\pi_\infty=\PG(2,p)$. A set $S$ determines a direction
if two distinct points of $S$ lie on a line of that direction. Write
$\NS_S$ for the set of directions not determined by $S$. A
\emph{cylinder} is a union of $p$ distinct parallel affine lines.
Ball's weak cylinder conjecture \cite{Ball08,Ball12} states that
a set of $p^2$ points with $|\NS_S|\ge p$ is a cylinder.

If $|S|=p^2$ and a direction is not determined, choosing it as the
vertical direction represents $S$ as the graph of a function
$f\colon\Fp^2\to\Fp$. We identify this function with its unique
polynomial representative of degree at most $p-1$ in each variable.
The degree considered below is the total degree of this representative
in the chosen coordinates. Our hypothesis is that $S$ admits such a
representation of bounded degree.

\begin{theorem}\label{thm:all-degrees}
Let $D\ge0$ and let $S\subseteq\AG(3,p)$ admit a graph representation
by a polynomial of total degree at most $D$. If $D\ge3$ and
$p>D^{9/2}$, then
\[
 |\NS_S|\ge p\qquad\Longrightarrow\qquad S\text{ is a cylinder}.
\]
For $D\le2$ the same conclusion holds for every odd prime.
\end{theorem}

For a non-cylindrical graph of actual degree $D\ge3$, we prove the
stronger estimate
\begin{equation}\label{eq:intro-bound}
 |\NS_S|\le p+1-\left\lceil\frac p{D-1}\right\rceil
 \qquad (p>D^{9/2});
\end{equation}
see Theorem~\ref{thm:quant}. In particular,
$|\NS_S|\le(1-1/(D-1))p+1$.
The bound in \eqref{eq:intro-bound} is attained whenever
$p\equiv1\pmod{D-1}$: for $f=x^{D-1}y$, Theorem~\ref{thm:sharp} gives
\[
 |\NS_S|=1+(p-1)\left(1-\frac1{\gcd(D-1,p-1)}\right).
\]
There are infinitely many such primes by Dirichlet's theorem.
For graphs of actual degree $D$ with $D$ even or prime, the cylinder conclusion already follows from
$p>D^4$; a degree-dependent bound is given in
Corollary~\ref{cor:threshold}.

The proof uses the difference quotients
\[
 A_v(x,y,t)=\frac{f(x+tv_1,y+tv_2)-f(x,y)}t,
 \qquad v=(v_1:v_2)\in\PP^1(\Fp),\quad t\ne0.
\]
The number of non-determined directions is one plus the sum, over $v$,
of the numbers of values omitted by $A_v$.
A value $\lambda$ is omitted precisely when the restriction of $f$ to
every affine line of direction $v$, after subtraction of $\lambda s$,
permutes $\Fp$. If these restrictions are affine, then
$f(u,s)=c_1(u)s+c_0(u)$ in adapted coordinates. Unless the graph is a
cylinder, every other direction has a surjective difference quotient.
If the restrictions have degree at least two, their difference
quotients are geometrically non-composite, and the number of omitted
values is at most $D-2$ for $p>D^4$. It remains to bound the number of
directions which contribute. The roots of the leading binary form
account for at most $D$ directions. For the other directions, the
restrictions are exceptional permutation polynomials of degree $D$.
Their classification, a reconstruction of their components, and a
rigidity lemma give a bound independent of $p$.

The direction problem originates in R\'edei's polynomial method
\cite{Redei}. The planar theory was developed in
\cite{LS,BBBSS,Ball03}; higher-dimensional polynomial and structural
results appear in \cite{BL04,BL06,Ball08,Sziklai06}, with surveys in
\cite{Ball12,DB14}.
In particular, Ball and Lavrauw \cite{BL04,BL06} prove that a set of
$p^2$ points with at least $p$ non-determined directions meets every
affine plane in a multiple of $p$ points. This divisibility conclusion
does not identify the set as a cylinder. Under a bound on the total
degree, Theorem~\ref{thm:all-degrees} gives that further conclusion,
and \eqref{eq:intro-bound} bounds the number of non-determined
directions of a non-cylindrical graph.

There is also a stronger version, the \emph{strong cylinder
conjecture}: a set of $p^2$ points of $\AG(3,p)$ meeting every plane in
$0\bmod p$ points is a cylinder. It implies the weak form, since
$|\NS_S|\ge p$ forces every plane to meet $S$ in $0\bmod p$ points
\cite[Theorem~1.3]{Ball08}; both conjectures are stated in
\cite[Conjectures~4.5 and~4.6]{DB14}.
De Beule, Demeyer, Mattheus and Sziklai \cite{DDMS} establish a relaxed
version for $p\le13$ under additional direction hypotheses. Kurz and
Mattheus \cite{KM} correct the earlier argument for $p=5$ and prove
the case $p=7$, using divisible codes, incidence counting and
enumeration; see \cite{Ward,Mattheus22}. Kiss, Mark\'o, Nagy and Somlai
\cite{KMarko} obtain span theorems for $p$-divisible sets and multisets,
which do not assert that such a set is a single cylinder. The results
of De Beule, Sziklai and Tak\'ats \cite{DST} concern sets of size
$q^{n-1}-\varepsilon$.

The missing-slope criterion in terms of permutation polynomials is
classical; see \cite{BG09,BGS08,BZ}. The results of Ball and G\'acs
\cite{BG09} concern degrees of reduced powers of a univariate
polynomial and the corresponding planar direction problem.
In the present setting the same omitted value must satisfy the
permutation condition on all $p$ parallel lines. Our argument uses
these simultaneous conditions to obtain polynomial identities in the
line parameter. The resulting rigidity statement, Lemma~\ref{lem:shear-rigidity},
applies also when the restrictions are decomposable. It concludes that
a fixed compositional type in a moving direction forces $f$ to be a
polynomial in one linear form, up to an affine term. This conclusion
differs from the preservation of indecomposability under generic line
restriction proved in \cite[Theorem~1.2]{BCD}, although the change of
variables and the differential identity have the same origin.

Theorem~\ref{thm:all-degrees} concerns total degree rather than
intersection multiplicities or the degrees of reduced powers.
It applies uniformly for $D=o(p^{2/9})$; it does not cover arbitrary
reduced graph polynomials, whose total degree can be $2(p-1)$.
The quantitative estimate \eqref{eq:intro-bound} also applies to
non-cylindrical graphs without assuming $|\NS_S|\ge p$.

Sections~\ref{sec:budget}--\ref{sec:pivot} give the geometric reduction
and the permutation criterion. Section~\ref{sec:count} proves the
rigidity lemma and counts the contributing directions.
Section~\ref{sec:size} bounds the number of values omitted in a fixed
direction, and Section~\ref{sec:sum} proves the main results.
An independent criterion for powers among the leading forms of the
difference quotients is retained in Appendix~\ref{sec:appendix}.

\medskip\noindent\textbf{Conventions.}
From Section~\ref{sec:budget} onwards $p$ is an odd prime and $D<p$.
A form is a homogeneous polynomial. A nonzero form is a perfect
$e$-th power up to a scalar if it equals $cU^e$, where $c\ne0$,
$e\ge2$ and $U$ is a nonconstant form. This condition is considered
over an algebraic closure unless a field is specified.
For variables $w_i$ of weight $i$, a polynomial is weighted homogeneous
of weight $j$ if each monomial $w_{i_1}\cdots w_{i_k}$ occurring in it
satisfies $i_1+\cdots+i_k=j$.
The variables $s$ and $u$ denote, respectively, the coordinate along a
line and the parameter indexing the parallel lines.

\section{Directions and difference quotients}\label{sec:budget}

Let $S\subseteq\AG(3,p)$ with $|S|=p^2$ and $\NS_S\ne\varnothing$, let
$m_0\in\NS_S$ and choose coordinates with $m_0=(0:0:1)$. Each of the
$p^2$ lines of direction $m_0$ meets $S$ in at most one point, hence by
counting in exactly one, so $S$ is the graph of a unique
$f\colon\Fp^2\to\Fp$, identified with its reduced representative.
Constant graphs, the zero graph included, are planes and hence
cylinders; from now on $f$ is nonconstant and $D:=\deg f\ge1$.
We assume $D<p$; this follows from the hypotheses of
Theorem~\ref{thm:all-degrees} in every nonconstant case under consideration.
The directions other than $m_0$ are the points $(v_1:v_2:\lambda)$.
For fixed $v=(v_1:v_2)\in\PP^1(\Fp)$, they lie on the projective line
\[
 \ell_v=\{(av_1:av_2:b):(a:b)\in\PP^1(\Fp)\}
 =\{m_0\}\cup\{(v_1:v_2:\lambda):\lambda\in\Fp\}.
\]

\begin{definition}\label{def:Av}
For $v\in\Fp^2\setminus\{0\}$ set
$A_v(x,y,t)=\bigl(f(x+tv_1,y+tv_2)-f(x,y)\bigr)/t$, a polynomial since
the numerator vanishes at $t=0$, and
\[
 \Val(A_v)=\{A_v(x,y,t):x,y\in\Fp,\ t\in\Fp^{*}\},
 \qquad \mu(v)=p-|\Val(A_v)| .
\]
For any polynomial $H(x,y,t)$ we use the same notation $\Val(H)$ for its
values on $\Fp^2\times\Fp^{*}$. The number $\mu(v)$ depends only on
$(v_1:v_2)$, and the value set scales, $\Val(A_{cv})=c\Val(A_v)$ for
$c\ne0$, because $A_{cv}(x,y,t)=cA_v(x,y,ct)$; we fix a representative
when writing $\Val(A_v)$.
\end{definition}

\begin{proposition}\label{prop:reduction}
$\NS_S\cap\ell_v=\{m_0\}\cup\{(v_1:v_2:\lambda):\lambda\notin\Val(A_v)\}$,
and therefore
\[
 |\NS_S|=1+\sum_{v\in\PP^1(\Fp)}\mu(v).
\]
\end{proposition}

\begin{proof}
{A pair of distinct points of $S$ whose difference has direction
$(v_1:v_2:\lambda)$, with $(v_1,v_2)\ne(0,0)$, is a pair $(P,f(P))$ and
$\bigl(P+t(v_1,v_2),f(P+t(v_1,v_2))\bigr)$ with $P\in\Fp^2$ and
$t\in\Fp^{*}$, and the two last coordinates differ by $t\lambda$.
Hence the direction $(v_1:v_2:\lambda)$ is determined by $S$ if and only
if $A_v(P,t)=\lambda$ for some such $P$ and $t$, that is, if and only if
$\lambda\in\Val(A_v)$. Together with $m_0\in\NS_S$ this is the
displayed equality.}

{For the formula, observe that the punctured pencils
$\ell_v\setminus\{m_0\}$, $v\in\PP^1(\Fp)$, partition
$\pi_\infty\setminus\{m_0\}$: a point of $\pi_\infty$ other than
$m_0$ is $(v_1:v_2:\lambda)$ with $(v_1,v_2)\ne(0,0)$, and it lies on
$\ell_v$ for exactly one $v=(v_1:v_2)\in\PP^1(\Fp)$. Since the
displayed equality puts $\NS_S\cap(\ell_v\setminus\{m_0\})$ in
bijection with $\Fp\setminus\Val(A_v)$, which has $\mu(v)$ elements,
summing over $v$ gives $|\NS_S|=1+\sum_v\mu(v)$.}
\end{proof}

Thus it suffices to prove $\sum_v\mu(v)<p-1$ for a non-cylindrical
graph. We first separate directions with affine restrictions, and then
bound both the number of the remaining contributing directions and
their individual contributions.

\section{Basic properties of the difference quotients}\label{sec:first-properties}

From now on, unless stated otherwise, $f$ is nonconstant of actual degree
$1\le D<p$. We write $f=f_D+f_{D-1}+\cdots$ and $g:=f_D\ne0$ for the
leading form, $g_x,g_y$ for its partial derivatives,
$\partial_v=v_1\partial_x+v_2\partial_y$ for the derivative in the
direction $v$, and $\widetilde A_v$ for the degree-$(D-1)$ homogeneous
component of $A_v$, which may be zero even when $A_v$ is not.

\begin{lemma}\label{lem:degree}
For $v\ne0$ we have $\deg A_v\le D-1$, and the degree-$(D-1)$ part of
$A_v$ is
\[
 \widetilde A_v(x,y,t)=\frac{g(x+tv_1,y+tv_2)-g(x,y)}{t}.
\]
Moreover:
\emph{(i)} $\widetilde A_v(x,y,0)=v_1g_x+v_2g_y$;
\emph{(ii)} the coefficient of $t^{D-1}$ in $\widetilde A_v$ equals $g(v)$;
\emph{(iii)} if $g=c\prod_{i=1}^{r}\ell_i^{m_i}$ over $\Fpbar$, with $c\ne0$ and pairwise
non-proportional linear forms $\ell_i$ and $c_i:=\ell_i(v)$, then
$t\widetilde A_v=c\bigl(\prod_i(\ell_i+c_it)^{m_i}-\prod_i\ell_i^{m_i}\bigr)$;
\emph{(iv)} $\widetilde A_v\equiv0$ if and only if $g=c\ell^{D}$ for a
linear form $\ell$ with $\ell(v)=0$.
\end{lemma}

\begin{proof}
By the Taylor expansion with Hasse derivatives,
\[
 f(x+u,y+w)-f(x,y)=\sum_{(i,j)\ne(0,0)}f^{(i,j)}(x,y)u^iw^j,
\]
with $\deg f^{(i,j)}\le D-i-j$. Putting $u=tv_1$, $w=tv_2$ and dividing
by $t$ gives
\[
 A_v=\sum_{k\ge1}t^{k-1}\sum_{i+j=k}f^{(i,j)}(x,y)v_1^iv_2^j .
\]
Each summand has degree at most $(k-1)+(D-k)=D-1$; the part of degree
exactly $D-1$ comes from $f_D$. Then (i) is $k=1$, (ii) follows from
$\widetilde A_v(0,0,t)=t^{-1}g(tv)=t^{D-1}g(v)$, and (iii) is immediate.
Finally, choose coordinates $(u,s)$ with $v=\partial_s$. If
$\widetilde A_v\equiv0$, then $\partial_sg=0$. Since $D<p$, this
implies $g\in\Fp[u]$, and homogeneity gives $g=cu^D$. The converse
follows directly from the displayed formula, proving (iv).
\end{proof}

\begin{lemma}\label{lem:cylinder}
$S$ is a cylinder if and only if $A_v$ is constant for some $v$; the
direction of the cylinder is then $(v_1:v_2:A_v)$, and it is never equal
to $m_0$.
\end{lemma}

\begin{proof}
Note that $A_v\equiv c$ means $f(P+tv)=f(P)+ct$ for all $P,t$, i.e.\ $S$
is the union of the $p$ lines with direction $(v_1:v_2:c)$.

Conversely, assume that $S$ is a cylinder, say the union of $p$ parallel
lines with direction $w$. Each of these lines is contained in $S$ and
carries $p\ge2$ points of $S$, so two distinct points of $S$ realise $w$;
that is, $w$ is determined by $S$. Since $m_0\in\NS_S$ is not determined,
$w\ne m_0$, and therefore $w=(v_1:v_2:c)$ with $v:=(v_1,v_2)\ne0$; we fix
this representative of $w$, which also fixes $c$. Let $P\in\Fp^2$. The
point $(P,f(P))$ lies on one of those $p$ lines, and that line is
contained in $S$, so $(P,f(P))+t(v_1,v_2,c)\in S$ for every $t\in\Fp$. As
$S$ is the graph of $f$, this says $f(P+tv)=f(P)+tc$, i.e.
\[
 A_v(P,t)=c\qquad\text{for all }P\in\Fp^2\text{ and all }t\in\Fp^{*} .
\]
It remains to extend this identity from $\Fp^2\times\Fp^{*}$ to $\Fp^3$
and then to an identity of polynomials. By Lemma~\ref{lem:degree} we have
$\deg(A_v-c)\le D-1$, so each of the degrees of $A_v-c$ in $x$, in $y$
and in $t$ is at most $D-1\le p-2$. Fix $a,b\in\Fp$: the univariate
polynomial $(A_v-c)(a,b,t)$ has degree at most $p-2$ and vanishes at the
$p-1$ elements of $\Fp^{*}$, hence it is the zero polynomial. Writing
$A_v-c=\sum_k h_k(x,y)t^k$, this means that every $h_k$ vanishes at every
point of $\Fp^2$; since $\deg_xh_k\le p-2$ and $\deg_yh_k\le p-2$, each
$h_k$ is the zero polynomial. Therefore $A_v-c=0$, so $A_v$ is the
constant $c$ and the cylinder has direction $(v_1:v_2:A_v)$.
\end{proof}

The following lemma applies to both $f$ and its leading form $g$.

\begin{lemma}\label{lem:unique-degenerate}
Let $F\in\Fp[x,y]$ with $\deg F<p$, and let $v,v'$ be non-proportional
directions with $\partial_v^2F\equiv\partial_{v'}^2F\equiv0$. Then
$\deg F\le2$. In particular, if $\deg F\ge3$ there is at most one
$v\in\PP^1(\Fp)$ with $\partial_v^2F\equiv0$.
\end{lemma}

\begin{proof}
After a linear change of coordinates the two directions are $(1,0)$ and
$(0,1)$, so $F_{xx}=F_{yy}=0$. Since $\deg F<p$, every monomial of $F$
has degree at most one in each variable, whence $\deg F\le2$.
\end{proof}

\section{The permutation criterion}\label{sec:pivot}

For $v\ne0$ let $\ell$ run over the $p$ affine $\Fp$-lines of $\AAA^2$ with direction
$v$, and for each of them fix a base point $P_\ell$, so that
$s\mapsto P_\ell+sv$ parametrises $\ell$. Set
\[
 \psi_{\ell,\lambda}(s):=f(P_\ell+sv)-\lambda s\in\Fp[s].
\]
Moving the base point to $P_\ell+s_0v$ replaces $\psi_{\ell,\lambda}(s)$
by $\psi_{\ell,\lambda}(s+s_0)+\lambda s_0$, so neither its degree nor
the property of permuting $\Fp$ depends on that choice.

It is convenient to use coordinates adapted to $v$. Fix $w\in\Fp^2$ such
that $(w,v)$ is a basis, and write each point of $\AAA^2$ uniquely as
$P=uw+sv$ with $u,s\in\Fp$. For $v=(0,1)$ one may take $w=(1,0)$, so that
$(u,s)=(x,y)$. Adding a multiple of $v$ changes only $s$, so the $p$
lines of direction $v$ are the levels $\{u=u_0\}$, $u_0\in\Fp$: the
coordinate $u$ labels them, and the derivative along $v$ is
$\partial_v=\partial_s$. As a function of $P$, $u$ is the linear form
with $u(v)=0$ and $u(w)=1$. Choosing $P_\ell=u_0w$, the point of $\ell$
with $s=0$, the point $P_\ell+sv$ has coordinates $(u_0,s)$, so the
parameter in $\psi_{\ell,\lambda}$ is the coordinate $s$ itself. Finally,
the change of coordinates is linear and invertible, so $f$ still has
total degree $D$ in $(u,s)$; writing $f=\sum_kc_k(u)s^{k}$, every
monomial $u^as^{k}$ occurring in $f$ satisfies $a+k\le D$, whence
$\deg c_k\le D-k$.

Whenever $\partial_vf\ne0$, put $r(v):=\max\{k\ge1:\partial_v^kf\ne0\}$.
Note that $\partial_vf=0$ yields $A_v=0$, and so $S$ is a cylinder by
Lemma~\ref{lem:cylinder}. In the adapted coordinates $r(v)=\deg_sf$,
that is, $r(v)$ is the degree of the generic restriction of $f$ to a line
of direction $v$. Indeed
$\partial_s^{k}f=\sum_{j\ge k}j(j-1)\cdots(j-k+1)\,c_j(u)s^{j-k}$, and
each falling factorial occurring here is a product of integers at most
$D<p$, hence nonzero in $\Fp$; so $\partial_s^kf\ne0$ precisely for
$k\le\deg_sf$. In particular $1\le r(v)\le D$, and $r(v)=D$ if and only
if $g(v)\ne0$, because the coefficient of $s^{D}$ in $f$ is the
coefficient of $s^{D}$ in $g$, namely $g(v)$. Then $\deg_tA_v=r(v)-1$:
by the expansion in the proof of Lemma~\ref{lem:degree} the coefficient
of $t^{k-1}$ in $A_v$ is the $k$-th Hasse derivative of $f$ in the
direction $v$, which equals $\partial_v^kf/k!$ because $k\le D<p$.
Moreover $\partial_v^{r(v)}f$ is constant on the lines of direction $v$,
being equal to $r(v)!\,c_{r(v)}(u)$, a nonzero polynomial in $u$ alone of
degree at most $D-r(v)$; so if $r(v)\ge2$, then
$\deg\psi_{\ell,\lambda}=r(v)$ for all but at most $D-r(v)$ lines: the
coefficient of $s^{r(v)}$ in $\psi_{\ell,\lambda}$ is
$c_{r(v)}(u_\ell)$, untouched by the correction $-\lambda s$ as soon as
$r(v)\ge2$, and it vanishes only at the at most $D-r(v)$ roots of
$c_{r(v)}$. For $r(v)=1$, subtracting $\lambda s$ can cancel the linear
term, so this degree assertion is not made: there $f=c_1(u)s+c_0(u)$ and
$\psi_{\ell,\lambda}=(c_1(u_\ell)-\lambda)s+c_0(u_\ell)$, which drops to
a constant on every line with $c_1(u_\ell)=\lambda$.

\begin{proposition}\label{prop:perm}
For $v\ne0$ and $\lambda\in\Fp$ one has $\lambda\notin\Val(A_v)$ if and
only if $\psi_{\ell,\lambda}$ permutes $\Fp$ for \emph{every} line $\ell$
of direction $v$.
\end{proposition}

\begin{proof}
A pair of distinct points of $S$ realising the direction
$(v_1:v_2:\lambda)$ is a pair $P_\ell+s_1v$, $P_\ell+s_2v$ on one line
$\ell$ with $s_1\ne s_2$ and
$f(P_\ell+s_1v)-f(P_\ell+s_2v)=\lambda(s_1-s_2)$, i.e.\
$\psi_{\ell,\lambda}(s_1)=\psi_{\ell,\lambda}(s_2)$. So $\lambda$ is
attained if and only if some $\psi_{\ell,\lambda}$ fails to be injective
and thus bijective. Every point of $\AAA^2$ lies on exactly one line of
direction $v$, so the $p$ polynomials $\psi_{\ell,\lambda}$ do account
for all such pairs.
\end{proof}

The criterion imposes a permutation condition on every line of the
given direction. We distinguish the possible degrees of these restrictions.

\begin{theorem}\label{thm:trichotomy}
Let $f$ have actual degree $3\le D<p$ and suppose that its graph $S$ is
not a cylinder. Let $v\in\PP^1(\Fp)$ with $\mu(v)>0$ and fix
$\lambda\notin\Val(A_v)$. Then $r(v)\ne2$, and exactly one of the
following holds.
\begin{enumerate}
\item[\textup{(L)}] $r(v)=1$, equivalently $\partial_v^2f\equiv0$
{(the equivalence uses $\partial_vf\ne0$, which holds because $S$ is
not a cylinder)}. Then
$f=c_1(u)s+c_0(u)$ and $A_v=c_1(u)$ involves neither $s$ nor $t$. At
most one direction is of this kind.
\item[\textup{(G)}] $3\le r(v)<D$, equivalently $g(v)=0$ and
$\partial_v^2f\ne0$.
\item[\textup{(A)}] $r(v)=D$, equivalently $g(v)\ne0$; then
$\psi_{\ell,\lambda}$ is a permutation polynomial of $\Fp$ of degree $D$
for \emph{every} line $\ell$ of direction $v$.
\end{enumerate}
Directions of kind \textup{(L)} or \textup{(G)} are roots of the binary
form $g$, hence at most $D$ in number. Writing $\BB_{\rm ar}(f)$ for the
set of directions of kind \textup{(A)} 
and $\BB(f)=\{v:\mu(v)>0\}$,
\[
 |\BB(f)|\le D+|\BB_{\rm ar}(f)| .
\]
\end{theorem}

\begin{proof}
Since $\mu(v)>0$ there is an omitted value $\lambda$, and
Proposition~\ref{prop:perm} makes every $\psi_{\ell,\lambda}$ a
permutation of $\Fp$. As $S$ is not a cylinder, $\partial_vf\ne0$ by
Lemma~\ref{lem:cylinder}, so $r(v)$ is defined and $1\le r(v)\le D$.

If $r(v)=2$, then by the degree count above $\psi_{\ell,\lambda}$ has
degree two on all but at most $D-2$ lines; since $p$ is odd, no such
polynomial permutes $\Fp$, a contradiction. Hence $r(v)\ne2$, and the
three cases are distinguished by the value of $r(v)$, with $r(v)=D$ if
and only if $g(v)\ne0$.

In case \textup{(L)} one has $\deg_sf=1$, that is $f=c_1(u)s+c_0(u)$,
whence $A_v(u,s,t)=\bigl(c_1(u)(s+t)-c_1(u)s\bigr)/t=c_1(u)$; moreover
$r(v)=1$ says $\partial_vf\ne0=\partial_v^{2}f$, and
Lemma~\ref{lem:unique-degenerate} applied to $F=f$ leaves at most one
such direction. In case \textup{(A)} we have $\deg c_D\le D-D=0$, so the
leading coefficient $c_D$ is the nonzero constant $g(v)$. For
$2\le r(v)<D$ the coefficient $c_{r(v)}(u)$ may vanish, and the degree
of the restriction then drops below $r(v)$ on the lines indexed by its
at most $D-r(v)$ roots; for $r(v)=D$ no such line exists, a nonzero
constant having no root. Hence $\deg\psi_{\ell,\lambda}=D$ on every
line of direction $v$. Finally
$r(v)<D$ means $g(v)=0$, and the binary form $g$ of degree $D$ has at
most $D$ roots in $\PP^1(\Fp)$.
\end{proof}

\begin{proposition}\label{prop:affine-restrictions}
Let $S$ be a non-cylindrical graph of degree $D<p$. Suppose that
$r(v)=1$ for some direction $v$. In coordinates adapted to $v$, write
$f(u,s)=c_1(u)s+c_0(u)$ and put $e=\deg c_1$. Then $1\le e\le D-1$,
every $A_w$ with $w\ne v$ is surjective, and
\[
 |\NS_S|=1+p-|c_1(\Fp)|
 \le p+1-\left\lceil\frac pe\right\rceil
 \le p+1-\left\lceil\frac p{D-1}\right\rceil<p.
\]
All non-determined directions lie on $\ell_v$.
\end{proposition}

\begin{proof}
We have $A_v=c_1(u)$ and $\deg c_1\le D-1$. If $c_1$ were constant,
Lemma~\ref{lem:cylinder} would make $S$ a cylinder, so $e\ge1$.
For a different direction $w=(a,b)$ in the coordinates $(u,s)$,
we have $a\ne0$, and the coefficient of $s$ in $A_w$ is
\[
 \frac{c_1(u+ta)-c_1(u)}t.
\]
Since $c_1$ is nonconstant and $e<p$, it is not constant as a function
on $\Fp$. Choose $u,u'\in\Fp$ with $c_1(u')\ne c_1(u)$ and set
$t=(u'-u)/a\ne0$. For these fixed $u,t$, the polynomial $A_w$ is
affine in $s$ with nonzero coefficient, hence assumes every value of
$\Fp$. Thus $\mu(w)=0$ for all $w\ne v$.
Proposition~\ref{prop:reduction} gives the equality and the assertion
about $\ell_v$. Each fibre of $c_1$ has at most $e$ points, so
$|c_1(\Fp)|\ge\lceil p/e\rceil$. Finally $e\le D-1<p$ implies
$\lceil p/e\rceil\ge2$, proving the strict inequality.
\end{proof}

\section{Exceptional restrictions and the number of directions}\label{sec:count}

By Theorem~\ref{thm:trichotomy}, it remains to bound
$|\BB_{\rm ar}(f)|$. For these directions every corrected line
restriction is a permutation polynomial of degree $D$. We use the
following two classical results.

\begin{itemize}
\item[\textup{(W)}] (Wan \cite{Wan}, with the explicit bound due to von
zur Gathen \cite{vZG}; see also Guralnick--Wan \cite{GW}, and the
improved bounds in \cite{CG,Fan}) for $p>D^4$ every permutation
polynomial $\psi$ of $\Fp$ of degree $n\le D$ is \emph{exceptional},
meaning that
$\Phi(s_1,s_2):=\bigl(\psi(s_1)-\psi(s_2)\bigr)/(s_1-s_2)$ has no
absolutely irreducible factor defined over $\Fp$. The case $n=1$ is immediate. For $n\ge2$, if
such a factor $F$ existed, of degree $d\le n-1$, the Weil bound for
absolutely irreducible plane curves \cite{AP} would give at least
$p+1-(d-1)(d-2)\sqrt p-2d$ affine $\Fp$-points on $F$, of
which at most $d$ lie on the diagonal $s_1=s_2$ (which is not a
component, since $\Phi(s,s)=\psi'(s)\not\equiv0$ as $n<p$); for
$p\ge n^4$ one has $\sqrt p\ge n^2>(n-2)(n-3)+3n/\sqrt p$, so a point
with $s_1\ne s_2$ and $\psi(s_1)=\psi(s_2)$ exists, contradicting
injectivity;
\item[\textup{(C)}] (Fried--Guralnick--Saxl \cite{FGS}, M\"uller
\cite[Appendix]{Muller}; see also \cite{LMT}) an exceptional polynomial
over $\Fp$ of degree $n$ with $\gcd(n,p)=1$ is a composition
$\alpha_0\circ\chi_1\circ\alpha_1\circ\cdots\circ\chi_k\circ\alpha_k$,
where the $\alpha_i$ are affine polynomials of degree one over $\Fp$ and
each $\chi_i$ is $x^{n_i}$ or a Dickson polynomial
$\Dick_{n_i}(x,\delta_i)$ with $n_i$ prime. The constituents are
exceptional. For $r$ prime to $p$, $x^r$ is exceptional exactly when
$\gcd(r,p-1)=1$; for $\delta\ne0$, $\Dick_r(x,\delta)$ is exceptional
exactly when $\gcd(r,p^2-1)=1$.
\end{itemize}

Our convention for Dickson polynomials is $\Dick_0(z,\delta)=2$,
$\Dick_1(z,\delta)=z$ and
$\Dick_r(z,\delta)=z\Dick_{r-1}(z,\delta)-\delta\Dick_{r-2}(z,\delta)$.
In particular $\Dick_r(z,0)=z^r$ for $r\ge1$.

A \emph{type} of degree $n$ is an ordered list
$\tau=(r_1,\ldots,r_k)$ of odd primes with product $n$.
For $r_i=3$ put $\chi_i(z)=z^3$; for $r_i\ge5$ put
$\chi_i(z)=\Dick_{r_i}(z,\delta_i)$, allowing $\delta_i=0$.
A polynomial $R\in F[z]$ \emph{has type $\tau$ over $F$} if
\[
 R=\alpha_0\circ\chi_1\circ\alpha_1\circ\cdots\circ\chi_k\circ\alpha_k
\]
for affine polynomials $\alpha_i\in F[z]$ of degree one and
parameters $\delta_i\in F$. The identity $\Dick_r(z,0)=z^r$
includes all power components of prime degree at least five.
The type does not specify the parameters and need not be realised
by an exceptional polynomial for a given prime.
For a polynomial in $K[u][x]$, the type is taken with respect to $x$
over $K(u)$, unless another field is specified.

{\begin{lemma}\label{lem:types}
Let $p$ be odd and let $\psi\in\Fp[s]$ be exceptional of degree $n$ with
$2\le n<p$. Then $n$ is odd and $\psi$ has one of the types of degree
$n$ over $\Fp$. In particular no polynomial of even degree in that range
is exceptional over $\Fp$.
\end{lemma}}

{\begin{proof}
Since $n<p$ we have $\gcd(n,p)=1$, so \textup{(C)} applies and writes
$\psi=\alpha_0\circ\chi_1\circ\alpha_1\circ\cdots\circ\chi_k\circ\alpha_k$
with the $\alpha_i$ of degree one over $\Fp$, each $\chi_i$ equal to
$x^{r_i}$ or to $\Dick_{r_i}(x,\delta_i)$ with $r_i$ prime, and every
$\chi_i$ itself exceptional. Fix $i$ and note that $r_i$ divides
$n$, so $r_i\ne p$. If $\chi_i=x^{r_i}$, or if
$\chi_i=\Dick_{r_i}(x,0)=x^{r_i}$, exceptionality means
$\gcd(r_i,p-1)=1$; as $p-1$ is even, $r_i$ is odd, hence $r_i\ge3$. If
$\chi_i=\Dick_{r_i}(x,\delta_i)$ with $\delta_i\ne0$, exceptionality
means $\gcd(r_i,p^2-1)=1$; as $p^2-1$ is even, $r_i$ is odd, and
$r_i=3$ is impossible: for $p\ne3$ because $3\mid p^2-1$, and for $p=3$
because $r_i\ne p$. Hence $r_i\ge5$. In every case $r_i$ is an odd
prime, so $n=r_1\cdots r_k$ is odd. Replacing every power
$x^{r_i}$ with $r_i\ge5$ by $\Dick_{r_i}(x,0)$ gives the type
$(r_1,\ldots,r_k)$ over $\Fp$.
\end{proof}}

Let $t(1)=1$, and let $t(n)$ be the number of types of degree $n$.
For $n>1$,
\[
 t(n)=\sum_{\substack{r\mid n\\r\text{ odd prime}}}t(n/r),
\]
so $t(n)=0$ for even $n$. If $n=\prod_r r^{a_r}$ is odd, then
$t(n)=(\sum_r a_r)!/\prod_r a_r!$.
Induction gives $t(n)\le n^{3/2}$, since
\[
 \sum_{r\text{ odd prime}}r^{-3/2}
 \le\sum_{j\ge0}(3+2j)^{-3/2}
 \le3^{-3/2}+\frac12\int_3^\infty x^{-3/2}\,dx
 =\frac4{3\sqrt3}<1.
\]

A priori the compositional type may vary with the line; the normal form
below records the coefficient bounds used in the subsequent arguments.

\begin{proposition}\label{prop:structure}
Let $D\ge3$ and $p>D^4$. Let $\lambda$ be omitted by the pencil $v$
and assume $g(v)\ne0$, so that $r(v)=D$. Then, in coordinates $(u,s)$
with $v=\partial_s$ and $u$ constant along $v$, there are $a\ne0$, an
affine $b(u)=\beta u+\gamma$ and polynomials $\chi_u$, $c$ such that
\begin{equation}\label{eq:normalform}
 f(u,s)=a\,\chi_u\bigl(s+b(u)\bigr)+\lambda s+c(u),
\end{equation}
where $\chi_u\in\Fp[u][z]$ is monic of $z$-degree $D$, its
coefficients of $z^{D-1}$ and $z^0$ vanish, and
$\deg_u[\chi_u]_j\le D-j$. For every $u_0\in\Fp$, the
specialisation $\chi_{u_0}$ is exceptional over $\Fp$.
\end{proposition}

\begin{proof}
Throughout we work in the adapted coordinates $(u,s)$ of
Section~\ref{sec:pivot}, so that $f=\sum_kc_k(u)s^{k}$ with
$\deg c_k\le D-k$, and we write $\psi_u:=\psi_{\ell_u,\lambda}$ for the
restriction attached to the line $\ell_u$ of direction $v$ through
$(u,0)$, that is $\psi_u(s)=f(u,s)-\lambda s$. For a univariate
polynomial $\chi$ we write $[\chi]_j$ for its coefficient of $z^{j}$.

\emph{Step 1: every restriction has degree $D$, with one and the same
leading coefficient.} The coefficient of $s^{D}$ in $f$ is $c_D$, and
$\deg c_D\le D-D=0$, so $c_D$ is the constant $g(v)=:a\ne0$. Since
$D\ge3$, the correction $-\lambda s$ does not affect the coefficients of
$s^{D}$ and $s^{D-1}$. Hence $\deg\psi_u=D$ for every $u\in\Fp$, with
leading coefficient $a$, and $[\psi_u]_{D-1}=c_{D-1}(u)$, a polynomial in
$u$ of degree at most $D-(D-1)=1$.

\emph{Step 2: every restriction is exceptional.} As $\lambda$ is omitted
by the pencil $v$, Proposition~\ref{prop:perm} makes $\psi_u$ a
permutation polynomial of $\Fp$ for every $u$. Its degree is $D$ and
$p>D^4$, so \textup{(W)} applies and $\psi_u$ is exceptional.
{Note that $D<p$, so $\gcd(D,p)=1$ and \textup{(C)} is available for
these polynomials as well.}

\emph{Step 3: the normalising substitution.} Since $D<p$, the integer
$D$ is invertible in $\Fp$, and we may set
\[
 b(u):=\frac{c_{D-1}(u)}{a\,D}\in\Fp[u],
\]
which is affine, $b=\beta u+\gamma$, by the degree bound of Step 1. Put
\[
 \Theta(u,z):=\frac1a\Bigl[f\bigl(u,z-b(u)\bigr)-\lambda\bigl(z-b(u)\bigr)\Bigr]
 \in\Fp[u,z],
\]
and then
\[
 \chi_u(z):=\Theta(u,z)-\Theta(u,0),\qquad c(u):=a\,\Theta(u,0).
\]
Substituting $z=s+b(u)$ in the definition of $\Theta$ gives
$a\,\Theta\bigl(u,s+b(u)\bigr)=f(u,s)-\lambda s$, whence
\[
 f(u,s)=a\,\chi_u\bigl(s+b(u)\bigr)+\lambda s+c(u),
\]
which is \eqref{eq:normalform}. We stress that this is an identity in
$\Fp[u,s]$: the coefficients of $\chi_u$ and the polynomial $c$ are
constructed as elements of $\Fp[u]$, not merely specialised at each
$u\in\Fp$.

\emph{Step 4: the normalised coefficients.}
Expanding $f(u,z-b(u))=\sum_kc_k(u)\bigl(z-b(u)\bigr)^{k}$ and using
$D\ge3$, which keeps the term $-\lambda(z-b)$ away from the degrees $D$
and $D-1$, we get
\[
 a\,[\Theta]_D=c_D=a,\qquad
 a\,[\Theta]_{D-1}=c_{D-1}(u)-D\,a\,b(u)=0
\]
by the choice of $b$: the first comes from $k=D$ alone, the second from
$k=D$ and $k=D-1$ only. Thus $\Theta$, and with it $\chi_u$, is monic of
degree $D$ in $z$ with zero coefficient of $z^{D-1}$, while
$\chi_u(0)=0$ holds by
construction.

\emph{Step 5: the degree bounds in $u$.} The polynomial $f$ has total
degree $D$ in $(u,s)$, and the substitution $s\mapsto z-b(u)$ sends a
monomial $u^{i}s^{k}$ to $u^{i}\bigl(z-b(u)\bigr)^{k}$, of total degree
at most $i+k$ because $\deg b\le1$; so $f(u,z-b(u))$ has total degree at
most $D$ in $(u,z)$. The term $\lambda(z-b(u))$ has total degree at most
$1\le D$. Hence $\deg\Theta\le D$, and the coefficient of $z^{j}$ in
$\Theta$ has degree at most $D-j$ in $u$. For $j\ge1$ that coefficient
is $[\chi_u]_j$, which gives $\deg_u[\chi_u]_j\le D-j$; for $j=0$ there
is nothing to prove since $[\chi_u]_0=0$.

\emph{Step 6: $\chi_u$ is exceptional for every $u\in\Fp$.} Fix
$u_0\in\Fp$ and put $b_0:=b(u_0)$. By the definition of $\Theta$ we have
$a\,\Theta(u_0,z)=\psi_{u_0}(z-b_0)$, so
\[
 \chi_{u_0}(z)=\alpha_0\circ\psi_{u_0}\circ\alpha_1(z),
 \qquad
 \alpha_1(z)=z-b_0,\quad
 \alpha_0(w)=\frac1a\bigl(w-\psi_{u_0}(-b_0)\bigr),
\]
a composition of $\psi_{u_0}$ with two polynomials of degree one over
$\Fp$. Exceptionality is unchanged by such compositions. By Step 2 the polynomial $\psi_{u_0}$ is
exceptional, hence so is $\chi_{u_0}$.
\end{proof}

\begin{remark}\label{rem:degrees}
{By Lemma~\ref{lem:types} the polynomials $\chi_{u_0}$, $u_0\in\Fp$, of
Proposition~\ref{prop:structure} have one of the $t(D)$ types of degree
$D$ over $\Fp$. If $D$ is even there is no such type, so $\BB_{\rm ar}(f)=\varnothing$. If $D$ is an odd prime,
$t(D)=1$: the type consists of $z^3$ for $D=3$, and of
$\Dick_D(z,\delta)$ for $D\ge5$, including $\delta=0$. If $D$
is an odd composite integer, decomposable types occur as well; the
argument below treats all odd degrees uniformly.}
\end{remark}

\subsection*{Rigidity of moving line restrictions}

{The normalisation of monic components carried out in
Lemma~\ref{lem:normal-components} is related to the approximate
roots of \cite{Bodin}, although we use no result from that paper: the
construction below is self-contained and proceeds by comparing
coefficients.} Bodin, Ch\`eze and D\`ebes \cite{BCD} study generic
restrictions $f(t+\alpha x,x)$ and differentiate the resulting polynomial
identities, following an identity of Wood \cite{Wood}; the operator
$\partial_m-x\partial_u$ below comes from the same change of variables.
The conclusion we need is that a fixed exceptional compositional type in
a moving direction forces a polynomial in one linear form, up to an
affine term. {Types are those defined before
Lemma~\ref{lem:types}.}

\begin{lemma}\label{lem:normal-components}
Let $R\in K[u][x]$ be monic of $x$-degree $n$ and total degree at most
$n$, where $\operatorname{char}K=0$ or $\operatorname{char}K>n$. Suppose
that, over an algebraic extension of $K(u)$,
\[
 R=\alpha_0\circ\chi_1\circ\alpha_1\circ\cdots\circ\chi_k\circ\alpha_k ,
\]
where each $\alpha_i$ has degree one and each $\chi_i$ is a power
$z^{r_i}$ with $r_i\ge3$ prime, or a Dickson polynomial
$\Dick_{r_i}(z,\delta_i)$ with $r_i\ge5$ prime. If its outer degree is
$r$ and $n=re$, it can be written
\[
 R=\Phi_r(H,\delta)+q(u),\qquad
 \Phi_r(z,\delta)=z^r\ \text{or}\ \Dick_r(z,\delta),
\]
with $H\in K[u][x]$ monic of $x$-degree $e$ and total degree at most $e$,
$q\in K[u]$, $\deg q\le n$, and, in the Dickson case, $\delta\in K[u]$,
$\deg\delta\le2e$. The same normalisation applies recursively to $H$.
\end{lemma}

\begin{proof}
\emph{Step 1.} Let $\Omega$ be the algebraic extension of $K(u)$ over
which $R$ has the given type. Its outermost factor is $\chi_1$, of degree
$r=r_1$, so that $n=re$ with $e=r_2\cdots r_k$; thus $\chi_1(z)=z^{r}$ or
$\chi_1(z)=\Dick_{r}(z,\delta')$ for some parameter $\delta'\in\Omega$,
and in both cases $\chi_1(z)=\Phi_r(z,\delta')$ with $\delta'=0$ in the
power case. Hence
\[
 R=\mu\,\Phi_r(B,\delta')+\nu,\qquad \mu\in\Omega^{*},\ \nu\in\Omega,
\]
where $B\in\Omega[x]$ gathers the inner factors $\chi_2,\dots,\chi_k$
together with the affine maps $\alpha_1,\dots,\alpha_k$; in particular
$\deg B=e$. Let $\sigma\ne0$ be the leading coefficient of $B$ and put
$H:=\sigma^{-1}B$, monic of degree $e$. In the power case
$B^{r}=\sigma^{r}H^{r}$, and in the Dickson case the weighted
homogeneity identity $\Dick_r(az,a^{2}\delta)=a^{r}\Dick_r(z,\delta)$
gives $\Dick_r(B,\delta')=\sigma^{r}\Dick_r(H,\sigma^{-2}\delta')$.
Setting $\delta:=\sigma^{-2}\delta'$, in both cases $\Phi_r(H,\delta)$ is
monic of degree $re=n$, so comparing leading coefficients with the monic
$R$ gives $\mu\sigma^{r}=1$ and
\[
 R=\Phi_r(H,\delta)+q,\qquad q:=\nu .
\]
So far $H$, $\delta$ and $q$ are only known to have coefficients in
$\Omega$.

\emph{Step 2.} Write $H=x^{e}+h_1x^{e-1}+\cdots+h_e$ and set $h_0:=1$.
Expanding the $r$-th power,
\[
 [x^{n-j}]H^{r}=\sum_{j_1+\cdots+j_r=j}h_{j_1}\cdots h_{j_r},
\]
the sum being over non-negative indices. The $r$ terms in which one index
equals $j$ and all the others vanish contribute $rh_j$, while every other
term involves only $h_1,\dots,h_{j-1}$. Hence
\[
 [x^{n-j}]H^{r}=r\,h_j+P_j(h_1,\dots,h_{j-1}),
\]
where $P_j$ has integer coefficients and, once $h_i$ is given the weight
$i$, is weighted homogeneous of weight $j$: the factors equal to $h_0=1$ carry
weight zero, and the indices of the remaining factors of each monomial
add up to $j$.

Compare this with $R$ for $1\le j\le e$. The term $q$ does not interfere,
because $n-j\ge n-e=(r-1)e>0$. In the Dickson case the terms of
$\Dick_r(H,\delta)$ other than $H^{r}$ are the $\delta^{i}H^{r-2i}$ with
$i\ge1$, of $x$-degree at most $(r-2)e$, and $n-j\ge(r-1)e>(r-2)e$, so
they do not interfere either. Therefore
\[
 [x^{n-j}]R=r\,h_j+P_j(h_1,\dots,h_{j-1}),\qquad 1\le j\le e .
\]
Since $r\le n$, the integer $r$ is invertible in $K$, and these relations
determine $h_1,\dots,h_e$ recursively,
\[
 h_j=\frac1r\Bigl([x^{n-j}]R-P_j(h_1,\dots,h_{j-1})\Bigr).
\]
Two consequences follow by induction on $j$. First $h_j\in K[u]$, because
$[x^{n-j}]R\in K[u]$ and $P_j$ has integer coefficients: this is what brings
$h_j$ from $\Omega$ back to $K$, and it requires no Galois argument. Second
$\deg h_j\le j$: the total degree of $R$ is at most $n$, so $[x^{n-j}]R$
has degree at most $j$ in $u$, and $P_j$ is weighted homogeneous of weight $j$, so
each of its monomials $h_{i_1}\cdots h_{i_k}$ with $i_1+\cdots+i_k=j$ has
degree at most $j$. Hence $H\in K[u][x]$ and its total degree is at most
$e$.

\emph{Step 3.} In the Dickson case, compare the coefficients of
$x^{n-2e}$. Among the $\delta^{i}H^{r-2i}$ only $i=0$ and $i=1$ reach
that degree, and $H^{r-2}$ is monic of degree $(r-2)e=n-2e$, so
\[
 [x^{n-2e}]R=[x^{n-2e}]H^{r}-r\delta ,
\]
again with no contribution from $q$, because $n-2e=(r-2)e>0$. Therefore
\[
 \delta=\frac1r\Bigl([x^{n-2e}]H^{r}-[x^{n-2e}]R\Bigr)\in K[u],
\]
and both terms have degree at most $2e$ in $u$, whence
$\deg\delta\le2e$. Finally $q=R-\Phi_r(H,\delta)$ lies in $K[u]$ as well,
and both $R$ and $\Phi_r(H,\delta)$ have total degree at most $n$ --- the
latter because $\Dick_r$ is weighted homogeneous of degree $r$ for the
weights $1$ and $2$, while $\deg H\le e$ and $\deg\delta\le2e$. Thus
$\deg q\le n$.

\emph{Step 4.} The polynomial $H$ is an affine rescaling of $B$, so it is
monic of degree $e$ and has the type
$\alpha\circ\chi_2\circ\alpha_2\circ\cdots\circ\chi_k\circ\alpha_k$ for a
suitable affine $\alpha$, that is, the given type with its outer factor
removed. If $e\ge3$ the lemma applies to $H$ and the normalisation
propagates inwards; if $e=1$ there is nothing to prove.
\end{proof}

For a monic polynomial of a fixed type, the preceding proof determines
the outer normalised decomposition $R=\Phi_r(H,\delta)+q$ uniquely.
Its \emph{outer additive term} means this polynomial $q$.
In particular, saying that this term vanishes means $q=0$ in that
normalised decomposition.

\begin{lemma}\label{lem:shear-rigidity}
Let $k$ be a field of characteristic zero or of characteristic
$p>n\ge3$, let $m$ be an indeterminate over $k$ and put $K=k(m)$. We
write $\partial_m$ for the usual derivative of $K$ with respect to $m$,
extended to $K[u]$ and to $K[x,u]$ by acting on the coefficients alone,
and $\partial_u$ for the derivative with respect to $u$. Let
$F_0\in k[x,y]$ have total degree at most $n$ and suppose that
\begin{equation}\label{eq:shear-family}
 F_0(x,u+mx)=a\,\Psi(x,u)+\lambda x+c(u),
\end{equation}
where $a\in K^{*}$, $\lambda\in K$, $c\in K[u]$, and $\Psi\in K[x,u]$ is
monic of $x$-degree $n$, has total degree at most $n$, has one of the
types above over an algebraic extension of $K(u)$ and has vanishing
outer additive term. Then there are a
nonzero linear form $L\in k[x,y]$, a polynomial $h\in k[z]$ and an
affine polynomial $\varphi\in k[x,y]$ such that
\[
 F_0=h(L)+\varphi .
\]
In particular, if $k=\Fp$, the graph of $F_0$ is a cylinder.
\end{lemma}

\begin{proof}
Put
\[
 \mathcal D:=\partial_m-x\,\partial_u .
\]
Being a difference of two operators that satisfy the Leibniz rule,
$\mathcal D$ satisfies it as well; it acts on $K$ as $\partial_m$, and
$\mathcal Dx=0$, $\mathcal Du=-x$.

\emph{Step 1: the operator $\mathcal D$ annihilates substituted
polynomials.} Let $F\in k[x,y]$ and $G(x,u):=F(x,u+mx)$. The chain rule
gives
\[
 \partial_mG=x\,F_y(x,u+mx),\qquad \partial_uG=F_y(x,u+mx),
\]
whence $\mathcal DG=0$. This is the only use we make of the particular
shape of the substitution $y=u+mx$: the left-hand side of
\eqref{eq:shear-family} is annihilated by $\mathcal D$.

\emph{Step 2: the differential identity.} Put
\[
 \rho:=\frac{\partial_ma}{n\,a}\in K .
\]
Apply $\mathcal D$ to \eqref{eq:shear-family}. On the left we get $0$ by
Step 1. On the right, $\mathcal D(a\Psi)=(\partial_ma)\Psi+a\,\mathcal D\Psi$
because $\partial_ua=0$, next $\mathcal D(\lambda x)=(\partial_m\lambda)x$
because $\partial_u\lambda=0$ and $\mathcal Dx=0$, and finally
$\mathcal D\bigl(c(u)\bigr)=\partial_mc-x\,\partial_uc$. Dividing the
resulting identity by $a$ and using $\partial_ma/a=n\rho$ we obtain
\begin{equation}\label{eq:scaled-shear}
 \mathcal D\Psi+n\rho\Psi
 =-\frac1a\Bigl[\bigl(\partial_m\lambda-\partial_uc\bigr)x
   +\partial_mc\Bigr].
\end{equation}
All we shall retain from \eqref{eq:shear-family} is that the right-hand
side has $x$-degree at most one; we shall prove that this condition forces $\Psi$ to depend on
one linear form. We also record \eqref{eq:scaled-shear} itself, which is used again
in Step 6.

\emph{Step 3: the two implications to be proved.} Let $R\in K[x,u]$
be monic of $x$-degree $N$, where $3\le N\le n$, of total degree at
most $N$, and of one of the types
above {over some algebraic extension of $K(u)$, which is the
hypothesis under which Lemma~\ref{lem:normal-components} applies and
returns components with coefficients in $K[u]$}. We prove, simultaneously and by induction on the number of
factors of the type, that for such an $R$
\begin{align}
 \mathcal DR+N\rho R=0&\ \Longrightarrow\ R\in K[x+\rho u],
 \label{eq:eigen-rigidity}\\
 \deg_x\bigl(\mathcal DR+N\rho R\bigr)\le1
 &\ \Longrightarrow\ R\in K[x+\rho u]\ \text{ and }\ \mathcal DR+N\rho R=0,
 \label{eq:low-rigidity}
\end{align}
the second implication under the additional hypothesis that the outer
additive term of $R$ vanishes.

\emph{Step 4: the case of prime degree.} Let $N=r$ be prime. By
Lemma~\ref{lem:normal-components} we may write
\[
 R=\Phi_r(z,\delta)+q(u),\qquad z=x+b(u),\quad b=\beta u+\gamma,
\]
with $\beta,\gamma\in K$ and $q\in K[u]$ of degree at most $r$; in the
power case $\Phi_r(z)=z^r$ with $r\ge3$ and there is no parameter, in
the Dickson case $\Phi_r(z,\delta)=\Dick_r(z,\delta)$ with $r\ge5$ and
$\delta\in K[u]$ of degree at most two. Since $(x,u)\mapsto(z,u)$ is a
triangular change of variables, $x$-degree and $z$-degree agree on
$K[x,u]=K[z,u]$, and we may compare coefficients of powers of $z$.

The polynomials $\Phi_r$ are weighted homogeneous of degree $r$ for the
weights $1$ on $z$ and $2$ on $\delta$, that is
$r\Phi_r=z\,\partial_z\Phi_r+2\delta\,\partial_\delta\Phi_r$. Together
with the chain rule for $\mathcal D$ this gives
\begin{equation}\label{eq:eigen-prime}
 \mathcal DR+r\rho R
 =\partial_z\Phi_r\,\bigl(\mathcal Dz+\rho z\bigr)
  +\partial_\delta\Phi_r\,\bigl(\mathcal D\delta+2\rho\delta\bigr)
  +\mathcal Dq+r\rho q ,
\end{equation}
the middle summand being absent in the power case. Note that
$\mathcal Dq+r\rho q$ has $z$-degree at most one, because
$\mathcal Dq=\partial_mq-x\,\partial_uq$ and $x=z-b(u)$.

From $\mathcal Dz=\partial_mb-x\,\partial_ub=\partial_mb-\beta x$ we get
\[
 \mathcal Dz+\rho z=(\rho-\beta)\,x+W(u),\qquad W:=\partial_mb+\rho b ,
\]
with $\deg W\le1$. Now read off the coefficient of $z^{r}$ in
\eqref{eq:eigen-prime}: the factor $\partial_z\Phi_r$ has leading term
$rz^{r-1}$ and $(\rho-\beta)x=(\rho-\beta)z-(\rho-\beta)b$, so that coefficient is
$r(\rho-\beta)$, the other two summands having $z$-degree at most $r-1$. It
vanishes under \eqref{eq:eigen-rigidity} because the whole expression
does, and under \eqref{eq:low-rigidity} because $r>1$. As $r\le n<p$ in
positive characteristic, $r$ is invertible in $K$ and therefore
\[
 \beta=\rho,\qquad \mathcal Dz+\rho z=W(u),\qquad z=x+\rho u+\gamma .
\]

In the power case \eqref{eq:eigen-prime} is
$rz^{r-1}W(u)+\mathcal Dq+r\rho q$. Since $r-1\ge2$ while the last two terms
have $z$-degree at most one, the coefficient of $z^{r-1}$ must vanish in
either implication, so $W=0$.

In the Dickson case the same computation applied to $\delta$ gives
\[
 \mathcal D\delta+2\rho\delta=-(\partial_u\delta)\,z+E(u),
 \qquad E:=\partial_m\delta+b\,\partial_u\delta+2\rho\delta .
\]
Insert this and the expansion
\[
 \Dick_r(z,\delta)=z^r-r\delta z^{r-2}
       +\frac{r(r-3)}2\,\delta^2z^{r-4}+\cdots
\]
into \eqref{eq:eigen-prime}. Since $r\ge5$, the three exponents $r-1$,
$r-2$, $r-3$ are at least two, so the corresponding coefficients receive
no contribution from $\mathcal Dq+r\rho q$ and must vanish in either
implication. Taken in this order they read
\[
 r\bigl(W+\partial_u\delta\bigr)=0,\qquad -rE=0,\qquad
 r\,\delta\,\partial_u\delta=0 ,
\]
the third one after substituting $W=-\partial_u\delta$ supplied by the
first. Hence $\delta\,\partial_u\delta=0$ in the integral domain
$K[u]$, so either $\delta=0$ or $\partial_u\delta=0$, and in both cases
$\partial_u\delta=0$. Since $\deg\delta\le2<p$ in positive
characteristic this gives $\delta\in K$; and then $W=0$ and
$\mathcal D\delta+2\rho\delta=0$.

In both cases $\Phi_r(z,\delta)$ has all its coefficients in $K$ and is a
polynomial in $z-\gamma=x+\rho u$, so that $R-q(u)\in K[x+\rho u]$, and
\eqref{eq:eigen-prime} reduces to
$\mathcal DR+r\rho R=\mathcal Dq+r\rho q$. Under \eqref{eq:low-rigidity} the
outer additive term $q$ vanishes by hypothesis, so $R\in K[x+\rho u]$ and
$\mathcal DR+r\rho R=0$, as asserted. Under \eqref{eq:eigen-rigidity} we have
$\mathcal Dq+r\rho q=0$; its coefficient of $x$ is $-\partial_uq$, so
$\partial_uq=0$, and $\deg q\le r<p$ in positive characteristic forces
$q\in K$. Again $R\in K[x+\rho u]$.

\emph{Step 5: the case of composite degree.} Let $N=re$ with $e\ge3$ and
write, again by Lemma~\ref{lem:normal-components},
$R=\Phi_r(H,\delta)+q(u)$, where $H$ is monic of $x$-degree $e$, has total degree at most $e$,
and has the given type with its outer factor removed, $\deg q\le N$ and,
in the Dickson case, $\delta\in K[u]$ with $\deg\delta\le2e$. Weighted
homogeneity, now in the form $N\rho=re\,\rho$, gives
\[
 \mathcal DR+N\rho R=\partial_z\Phi_r(H,\delta)\,\bigl(\mathcal DH+e\rho H\bigr)
 +\partial_\delta\Phi_r(H,\delta)\,\bigl(\mathcal D\delta+2e\rho\delta\bigr)
 +\mathcal Dq+N\rho q .
\]
Compare $x$-degrees. The factor $\partial_z\Phi_r(H,\delta)$ has
$x$-degree $(r-1)e$ and $\partial_\delta\Phi_r(H,\delta)$ has $x$-degree
$(r-2)e$, while $\mathcal D\delta+2e\rho\delta$ and $\mathcal Dq+N\rho q$ have
$x$-degree at most one. So if $\mathcal DH+e\rho H\ne0$, the first summand
has $x$-degree at least $(r-1)e$, whereas the other two have $x$-degree
at most $(r-2)e+1$, which is smaller because $e\ge3>1$. No cancellation
is then possible and the total $x$-degree is at least $(r-1)e>1$,
contradicting either implication. Therefore $\mathcal DH+e\rho H=0$, and
\eqref{eq:eigen-rigidity} applied to $H$ gives $H\in K[x+\rho u]$.

In the Dickson case the same comparison forces
$\mathcal D\delta+2e\rho\delta=0$: otherwise the second summand alone has
$x$-degree at least $(r-2)e\ge9$, which the other two cannot cancel. Its
coefficient of $x$ is $-\partial_u\delta$, so $\partial_u\delta=0$, and
$\deg\delta\le2e<N<p$ gives $\delta\in K$. Hence
$R-q(u)=\Phi_r(H,\delta)\in K[x+\rho u]$, and one concludes exactly as in
Step 4: under \eqref{eq:low-rigidity} one has $q=0$ and the whole
expression vanishes, while under \eqref{eq:eigen-rigidity} the identity
$\mathcal Dq+N\rho q=0$ gives $\partial_uq=0$ and then $q\in K$, because
$\deg q\le N<p$. This completes the induction.

\emph{Step 6: the conclusion over $K$.} Apply \eqref{eq:low-rigidity} to
$R=\Psi$ and $N=n$, which is legitimate by \eqref{eq:scaled-shear} and
because the outer additive term of $\Psi$ vanishes. We obtain
$\mathcal D\Psi+n\rho\Psi=0$ and $\Psi=\psi(x+\rho u)$ for some $\psi\in K[z]$.
Substituting the first identity into \eqref{eq:scaled-shear} leaves
\[
 \bigl(\partial_m\lambda-\partial_uc\bigr)x+\partial_mc=0
\]
as an identity in $K[x,u]$. Its coefficients of $x^{0}$ and $x^{1}$ give
$\partial_mc=0$ and $\partial_uc=\partial_m\lambda\in K$. Setting $x=0$
in \eqref{eq:shear-family} shows $\deg c\le n$, so $\partial_uc\in K$
together with $\deg c\le n<p$ forces $c=\alpha u+\beta$ with
$\alpha,\beta\in K$. Substituting $u=y-mx$ in \eqref{eq:shear-family} we
find
\[
 F_0=a\,\psi\bigl((1-\rho m)x+\rho y\bigr)+(\lambda-\alpha m)x+\alpha y+\beta ,
\]
a polynomial in the linear form $(1-\rho m)x+\rho y$ plus an affine polynomial,
with all coefficients in $K$. That linear form is nonzero: if $\rho=0$ its
coefficient of $x$ equals $1$.

\emph{Step 7: removing the indeterminate $m$.} Step 6 produced a linear
form over $K$, not over $k$, and we now replace it by one over $k$. For
$v=(v_1,v_2)$ consider the condition that
$v_1\,\partial_xF_0+v_2\,\partial_yF_0$ be a constant, that is, that all
its coefficients on monomials of positive degree vanish. As
$F_0\in k[x,y]$, expanding $\partial_xF_0$ and $\partial_yF_0$ in the
monomial basis turns this into a homogeneous system of linear equations
in $v_1,v_2$ whose coefficients lie in $k$, one equation per monomial of
positive degree. By Step 6 the pair $(v_1,v_2)=(\rho,\rho m-1)$ solves that
system over $K$: it is nonzero, and it annihilates the linear form
$(1-\rho m)x+\rho y$, since $(1-\rho m)\rho+\rho(\rho m-1)=0$, so that
$v_1\partial_xF_0+v_2\partial_yF_0$ reduces to the corresponding
derivative of the affine part. A homogeneous linear system with
coefficients in $k$ that has a nonzero solution in an extension field of
$k$ has a nonzero solution in $k$ itself, because its rank is computed
by the same minors over either field. Hence there is
$0\ne(v_1,v_2)\in k^{2}$ with $v_1\partial_xF_0+v_2\partial_yF_0=:c_0$
constant, and $c_0\in k$.

Choose linear coordinates $(X,Y)$ over $k$ in which this direction is
$\partial_Y$. Then $\partial_Y(F_0-c_0Y)=0$, which in characteristic
zero means $F_0-c_0Y\in k[X]$, and in characteristic $p$ means
$F_0-c_0Y\in k[X,Y^{p}]$; since $\deg_YF_0\le n<p$, the two conclusions
coincide. Thus $F_0=h(X)+c_0Y$ with $h\in k[z]$, which is the assertion
with $L=X$ and $\varphi=c_0Y$. The final sentence is
Lemma~\ref{lem:cylinder}.
\end{proof}

\subsection*{The uniform bound}

It remains to count the directions of kind \textup{(A)}. The count is
effective: a canonical reconstruction of the monic components reduces it
to the roots of univariate coefficient polynomials.

\begin{lemma}\label{lem:reconstruction}
Fix a type $\tau$ of degree $n\ge3$, and let
$P=x^n+z_1x^{n-1}+\cdots+z_n$ be monic over a field of characteristic
zero or greater than $n$. There is a canonical reconstructed polynomial
$R_\tau(P)$ with the following properties.
\begin{enumerate}
\item[\textup{(i)}] It is monic of degree $n$, is of type $\tau$, and its
outer additive term is zero. Its coefficient of $x^{n-j}$ is a polynomial
in $z_1,\ldots,z_{n-2}$, obtained from a universal polynomial over
$\mathbb Z[1/n!]$ by reduction in positive characteristic, weighted homogeneous of
weight $j$ when $z_i$ is assigned weight $i$.
\item[\textup{(ii)}] If $P-\eta x$ is of type $\tau$ for some scalar
$\eta$, then $P-R_\tau(P)=\eta x+c$ for a scalar $c$.
\item[\textup{(iii)}] If $z_j\in K[u]$ and $\deg_uz_j\le j$, then
$R_\tau(P)$ has total degree at most $n$ in $(x,u)$, and its successive
monic components have total degree at most their $x$-degree.
\item[{\textup{(iv)}}] {The construction commutes with
specialisation of the coefficients. Precisely, the coefficients of
$R_\tau(P)$ are universal polynomials in $z_1,\ldots,z_{n-2}$ over
$\mathbb Z[1/n!]$ by \textup{(i)}, so if $A$ is any commutative ring in
which $n!$ is invertible, $P$ is monic over $A$ and $\sigma\colon A\to
A'$ is a ring homomorphism into a ring in which $n!$ is invertible, then
$R_\tau(P^{\sigma})=\bigl(R_\tau(P)\bigr)^{\sigma}$, where $\sigma$ is
applied coefficientwise.}
\end{enumerate}
\end{lemma}

\begin{proof}
Let $r$ be the outer degree, $n=re$. Reconstruct
$H=x^e+h_1x^{e-1}+\cdots+h_e$ by requiring $H^r$ and $P$ to have the same
coefficients of $x^{n-1},\ldots,x^{n-e}$. The successive equations have
the form $rh_j=z_j-\text{a polynomial in the preceding }h_i$. Thus $h_j$
is weighted homogeneous of weight $j$ in $z_1,\ldots,z_j$. For a Dickson outer
factor put
\[
 \delta=\frac{[x^{n-2e}]H^r-z_{2e}}{r} .
\]
It is weighted homogeneous of weight $2e$. This formula is used only for $r\ge5$.

If $e=1$, put $\widehat H=H$. Otherwise, let $\tau'$ be the remaining
type and put
\[
 \widehat H=R_{\tau'}(H)+H(0)-R_{\tau'}(H)(0).
\]
Define $R_\tau(P)=\Phi_r(\widehat H,\delta)$, omitting $\delta$ in the
power case. This is a polynomial of the stated type for every input $P$.
Induction shows that its coefficient of $x^{n-j}$ is weighted homogeneous of
weight $j$. All
divisions are by outer prime degrees. Moreover, $e\le n-2$, and
$2e\le n-2$ whenever the parameter is used; therefore the construction is
independent of $z_{n-1}$ and $z_n$.

If $P-\eta x$ has type $\tau$, normalise its actual outer decomposition
as in Lemma~\ref{lem:normal-components}. The reconstructed $H$ and
$\delta$ are its actual monic inner component and parameter, since
neither $\eta x$ nor the outer additive constant affects the coefficients
used. The inner component has type $\tau'$. Applying (ii) inductively
with $\eta=0$, its reconstruction, with its constant coefficient
restored, is exactly $H$.
Thus $\widehat H=H$ and (ii) follows. This argument also applies when the
initial decomposition is given over an algebraic extension. Finally the
weight bounds give (iii), including the bounds on every intermediate
component and parameter.
\end{proof}

For $D\ge3$ define
\begin{equation}\label{eq:BD}
 B(D)=\begin{cases}
 0,&D\text{ even},\\
 t(D)(D-1)^2+1,&D\text{ odd}.
 \end{cases}
\end{equation}

\begin{theorem}\label{thm:uniform-bad}
Let $D\ge3$ and $p>D^4$, and let $f\in\Fp[x,y]$ have total degree
$D$ and non-cylindrical graph. Then
\[
 |\BB_{\rm ar}(f)|\le B(D)
 \le D^{3/2}(D-1)^2+1.
\]
\end{theorem}

\begin{proof}
The even case is Remark~\ref{rem:degrees}, so assume $D$ odd. Put
\[
 a(m)=f_D(1,m),\qquad
 P_m(x,u)=\frac{f(x,u+mx)}{a(m)}=x^D+\sum_{j=1}^{D}z_j(m,u)x^{D-j}.
\]
We work on $a(m)\ne0$: every direction $v=(1:m)$ of kind
\textup{(A)} satisfies this condition by Theorem~\ref{thm:trichotomy}. Write
$z_j=B_j(m,u)/a(m)$. Expanding each monomial
of $f$ shows
\begin{equation}\label{eq:coefficient-bidegrees}
 \deg_uB_j\le j,\qquad \deg_mB_j\le D-j,\qquad \deg_ma\le D .
\end{equation}
For each type $\tau$, set $E_\tau=P_m-R_\tau(P_m)$. The identities
required for a pencil are: all coefficients of $x^i$ in $E_\tau$ vanish
for $i\ge2$, and the coefficient of $x$ is independent of $u$. The
constant coefficient is unrestricted. The coefficient of $x^{D-j}$ in
$E_\tau$ has $u$-degree at most $j$.

{We shall pass freely between the generic $P_m$ and its
specialisations $P_{m_0}(x,u_0)$, which is legitimate by
Lemma~\ref{lem:reconstruction}\textup{(iv)}: $D!$ is invertible in
$\Fp$ because $p>D^4>D$, and evaluation at $m=m_0$, $u=u_0$ is a ring
homomorphism defined on the localization $\Fp[m,u,a(m)^{-1}]$.
Consequently $E_\tau$ specialises to the corresponding expression formed
from $P_{m_0}(x,u_0)$.}

Fix a direction of kind \textup{(A)}, of slope $m_0$, and a missing value
$\lambda$. For every $u_0\in\Fp$, the polynomial
$P_{m_0}(x,u_0)-\lambda x/a(m_0)$ is exceptional by \textup{(W)} and has
at least one of the $t(D)$ types. Choose one such type for each $u_0$.
Since $p>D^4>(D-1)t(D)$, one type occurs for more than $D-1$ values of
$u_0$. By Lemma~\ref{lem:reconstruction}(ii), at those values every
coefficient of $x^i$, $i\ge2$, in $E_\tau$ is zero, and its coefficient
of $x$ is $\lambda/a(m_0)$. Each of these polynomial identities has
$u$-degree at most $D-1$, so interpolation makes it an identity in $u$.
Thus $m_0$ satisfies all required identities for one type.

We now count the possible $m_0$. For $1\le j\le D-1$, multiply the
coefficient of $x^{D-j}$ in $E_\tau$ by $a(m)^j$.
Lemma~\ref{lem:reconstruction}(i) and \eqref{eq:coefficient-bidegrees}
show that this is a polynomial in $m,u$ of $m$-degree at most $j(D-1)$.
Indeed a monomial $z_{i_1}\cdots z_{i_h}$ of weight $i_1+\cdots+i_h=j$
has numerator, over $a^j$, $a^{j-h}B_{i_1}\cdots B_{i_h}$, of degree at
most $(j-h)D+hD-j=j(D-1)$. The same calculation applies to $z_j$ itself.
After taking coefficients in $u$, all required identities are therefore
polynomial equations in $m$ of degree at most $(D-1)^2$.

For any fixed type these equations cannot all be identically zero.
Otherwise over $K=\Fp(m)$ one would have
\[
 f(x,u+mx)=a(m)R_\tau(P_m)(x,u)+\lambda(m)x+c(m,u),
\]
with $\lambda(m)\in K$ and $c(m,u)\in K[u]$.
Lemma~\ref{lem:reconstruction}(iii) supplies all the degree conditions
required by Lemma~\ref{lem:shear-rigidity}, which applies with $k=\Fp$
and $K=\Fp(m)$. The graph would be a cylinder, a contradiction. Hence one nonzero
equation bounds the directions of each type by $(D-1)^2$. Summing over
the types and adding the direction outside the chart $v=(1:m)$ proves the
result.
\end{proof}

\section{Bounds for omitted values}\label{sec:size}

We retain the notation of Sections~\ref{sec:budget}--\ref{sec:pivot}.
A nonconstant polynomial $A\in k[x_1,\ldots,x_n]$ is
\emph{geometrically non-composite} if it cannot be written as $B(C)$
over $\bar k$, with $B$ univariate of degree at least two and $C$ a
polynomial in $x_1,\ldots,x_n$.

\begin{lemma}\label{lem:quotient-noncomposite}
Let $k$ have characteristic zero or characteristic $p>D$, and let
$f(u,s)\in k[u,s]$ have total degree at most $D$ and $s$-degree
$r\ge2$. Then
\[
 A(u,s,t)=\frac{f(u,s+t)-f(u,s)}t
\]
is geometrically non-composite. In particular, for the polynomials
considered here, $r(v)\ge2$ implies that $A_v$ is geometrically
non-composite.
\end{lemma}

\begin{proof}
Write $f(u,s)=\sum_{j=0}^r c_j(u)s^j$, where $c_r\ne0$.
Regard $A$ as a polynomial in $s,t$ over $\bar k(u)$.
Its homogeneous part of degree $r-1$ in these two variables is
\[
 c_r(u)\frac{(s+t)^r-s^r}{t}
 =c_r(u)\prod_{\substack{\zeta^r=1\\\zeta\ne1}}
       \bigl(t+(1-\zeta)s\bigr).
\]
The factors are distinct because the characteristic does not divide
$r$. If $A=B(C)$ with $\deg B=e\ge2$, then $C$ has positive
degree $d$ in $s,t$, and the highest homogeneous part of $A$ would
be $b_eC_d^e$, where $b_e$ is the leading coefficient of $B$.
This contradicts the absence of repeated factors in the displayed
form. Thus no such decomposition exists, even over $\bar k(u)$.
\end{proof}

For a geometrically non-composite polynomial $Q$ in two variables,
the Stein--Lorenzini inequality gives
\[
 |\spec(Q)|\le\deg Q-1,
 \qquad
 \spec(Q)=\{\lambda\in\Fpbar:Q-\lambda
       \text{ is reducible over }\Fpbar\}.
\]
The inequality holds in arbitrary characteristic; see
\cite{Stein,Lor,BDN}. It extends to more than two variables by
\cite{Najib}, although only the two-variable case is needed here.

\begin{proposition}\label{prop:small-pencil}
Let $3\le D<p$, let $f\in\Fp[x,y]$ have total degree $D$, and assume
$p>D^4$. If $r(v)\ge2$, then
\[
 \mu(v)\le r(v)-2\le D-2.
\]
\end{proposition}

\begin{proof}
In coordinates adapted to $v$, write
$f(u,s)=\sum_{j=0}^r c_j(u)s^j$, where $r=r(v)$ and $c_r\ne0$.
Since $\deg c_r\le D-r<p$, there is $u_0\in\Fp$ with
$c_r(u_0)\ne0$. Put $h(s)=f(u_0,s)$ and
\[
 Q(s,t)=\frac{h(s+t)-h(s)}t.
\]
Then $\deg Q=r-1$, and Lemma~\ref{lem:quotient-noncomposite}, applied
to $h$, shows that $Q$ is geometrically non-composite. Every value
omitted by $A_v$ is also omitted by $Q$ on $\Fp\times\Fp^*$.

Suppose that $Q-\lambda$ is absolutely irreducible, and put $d=r-1$.
Its projective closure has degree $d$. The Weil bound for possibly
singular plane curves \cite{AP}, after removing the points at infinity,
gives at least $p+1-(d-1)(d-2)\sqrt p-2d$ affine rational points.
Moreover, $Q(s,0)=h'(s)$ is nonconstant, since $2\le r<p$.
Thus $Q(s,0)-\lambda$ has at most $d$ roots. There are at least
\[
 p+1-(d-1)(d-2)\sqrt p-3d
\]
points on $Q-\lambda=0$ with $t\ne0$. This number is positive for
$p>D^4$: writing $z=\sqrt p>D^2$ and using $d\le D-1$, we have
\[
 z\bigl(z-(d-1)(d-2)\bigr)
 >D^2\bigl(D^2-(D-2)(D-3)\bigr)
 =D^2(5D-6)>3(D-1).
\]
Consequently every omitted value lies in $\spec(Q)$. The
Stein--Lorenzini inequality gives
$\mu(v)\le|\spec(Q)|\le r-2$.
\end{proof}

\section{Proof of the main result}\label{sec:sum}

\begin{lemma}\label{lem:lowdegree}
Let $p$ be odd and let $S$ have a graph representation $f$ with
$\deg f\le2$. Then either $S$ is a cylinder or $\NS_S=\{m_0\}$.
\end{lemma}

\begin{proof}
If $\deg f\le1$, its graph is a plane and hence a cylinder. Suppose that
$\deg f=2$ and write $f=q+\ell+c$, with $q$ homogeneous of degree two and
$\ell$ linear. Then
\[
 A_v(P,t)=\partial_vq(P)+q(v)t+\partial_v\ell .
\]
If the linear form $\partial_vq(P)$ is nonzero, then $A_v$ is surjective
as $P$ varies. If $\partial_vq=0$, then also $q(v)=0$, so
$A_v=\partial_v\ell$ is constant and $S$ is a cylinder by
Lemma~\ref{lem:cylinder}. Thus every pencil is surjective unless $S$ is a
cylinder, and the claim follows from Proposition~\ref{prop:reduction}.
\end{proof}

Put
\begin{equation}\label{eq:CD}
 C(D)=1+(D+B(D))(D-2),
\end{equation}
where $B(D)$ is defined in \eqref{eq:BD}.

\begin{theorem}\label{thm:quant}
Let $f\in\Fp[x,y]$ have actual degree $D\ge3$ and non-cylindrical
graph $S$, and assume $p>D^4$. If $r(v)\ge2$ for every direction $v$,
then $|\NS_S|\le C(D)$. Otherwise, in coordinates adapted to the
unique direction with $r(v)=1$, we have
$f(u,s)=c_1(u)s+c_0(u)$ and
\[
 |\NS_S|=1+p-|c_1(\Fp)|
 \le p+1-\left\lceil\frac p{D-1}\right\rceil.
\]
In particular, for $p>D^{9/2}$ every such graph satisfies
\begin{equation}\label{eq:quant-bound}
 |\NS_S|\le p+1-\left\lceil\frac p{D-1}\right\rceil.
\end{equation}
\end{theorem}

\begin{proof}
If $r(v)=1$ for some $v$, the assertions follow from
Lemma~\ref{lem:unique-degenerate} and
Proposition~\ref{prop:affine-restrictions}. Otherwise, the
non-cylindrical hypothesis excludes $\partial_vf=0$, so $r(v)\ge2$
for every $v$. By Theorems~\ref{thm:trichotomy} and
\ref{thm:uniform-bad}, at most $D+B(D)$ directions contribute.
Each contributes at most $D-2$ by Proposition~\ref{prop:small-pencil}.
Proposition~\ref{prop:reduction} therefore gives $|\NS_S|\le C(D)$.

Since $t(D)\le D^{3/2}$, we have
\[
 D+B(D)\le D+D^{3/2}(D-1)^2+1<D^{7/2}
 \qquad(D\ge3),
\]
the last inequality following from $D^{3/2}(2D-1)>D+1$.
Put $\alpha=(D-2)/(D-1)$. Then
\[
 \alpha D^{9/2}-\bigl(1+(D-2)D^{7/2}\bigr)
 =\frac{D-2}{D-1}D^{7/2}-1>0.
\]
Thus $p>D^{9/2}$ implies $C(D)<\alpha p$, whereas
$p+1-\lceil p/(D-1)\rceil\ge\alpha p$.
This proves \eqref{eq:quant-bound} also in the case without a
direction of restriction degree one.
\end{proof}

\begin{corollary}\label{cor:threshold}
Let $D\ge3$. If $p>\max\{D^4,C(D)\}$, then every graph of actual
degree $D$ with $|\NS_S|\ge p$ is a cylinder. In particular, $p>D^4$
suffices when $D$ is even or prime. If $p>D^4$ and a non-cylindrical
graph of actual degree $D$ satisfies $|\NS_S|>C(D)$, then it has a
representation $f(u,s)=c_1(u)s+c_0(u)$, and all its non-determined
directions lie on one projective line.
\end{corollary}

\begin{proof}
Theorem~\ref{thm:quant} gives $|\NS_S|<p$ in the first case if
$p>C(D)$; Proposition~\ref{prop:affine-restrictions} gives it in the
second case. If $D$ is even, then $B(D)=0$ and
$C(D)=1+D(D-2)<D^4$. If $D$ is an odd prime, then $t(D)=1$ and
$C(D)=D^3-3D^2+4D-3<D^4$.
The final assertion follows from the first alternative of
Theorem~\ref{thm:quant} and Proposition~\ref{prop:affine-restrictions}.
\end{proof}

\begin{proof}[Proof of Theorem~\ref{thm:all-degrees}]
Constant graphs are planes. Degrees one and two are covered by
Lemma~\ref{lem:lowdegree}. Suppose that the actual degree is
$3\le d\le D$. If the graph is not a cylinder and $p>D^{9/2}$,
then $p>d^{9/2}$ and Theorem~\ref{thm:quant} gives
\[
 |\NS_S|\le p+1-\left\lceil\frac p{d-1}\right\rceil\le p-1,
\]
because $p>d-1$. Hence $|\NS_S|\ge p$ forces $S$ to be a cylinder.
\end{proof}

\begin{theorem}\label{thm:sharp}
Let $D\ge3$, $p>D$, $f(x,y)=x^{D-1}y$ and let $S$ be its graph. Then $S$
is not a cylinder and
\[
 |\NS_S|=1+(p-1)\Bigl(1-\frac{1}{\gcd(D-1,\,p-1)}\Bigr).
\]
If $(D-1)\mid(p-1)$, the right-hand side is
$1+(p-1)\frac{D-2}{D-1}$.
\end{theorem}

\begin{proof}
Let $v^{*}=(0,1)$. Then
$A_{v^{*}}=\bigl(x^{D-1}(y+t)-x^{D-1}y\bigr)/t=x^{D-1}$, so
$\Val(A_{v^{*}})=\{x^{D-1}:x\in\Fp\}$ and
\[
 |\Val(A_{v^{*}})|=1+\frac{p-1}{\gcd(D-1,p-1)},
 \qquad
 \mu(v^{*})=(p-1)\Bigl(1-\frac1{\gcd(D-1,p-1)}\Bigr).
\]
For every other direction, write $v=(1,s)$. Setting $x=0$ and $t=1$ gives
$A_{(1,s)}(0,y,1)=y+s$, so $A_{(1,s)}$ is surjective and $\mu(1,s)=0$.
Proposition~\ref{prop:reduction} gives the formula. Finally no $A_v$ is
constant, and Lemma~\ref{lem:cylinder} shows that $S$ is not a cylinder.
\end{proof}

\begin{remark}\label{rem:L-is-extremal}
When $p\equiv1\pmod{D-1}$, the value in Theorem~\ref{thm:sharp} equals
$p+1-\lceil p/(D-1)\rceil$. Thus \eqref{eq:quant-bound} is attained
for every such prime in its range of validity.
More generally, Proposition~\ref{prop:affine-restrictions} determines
the number of non-determined directions for every non-cylindrical
graph $f(u,s)=c_1(u)s+c_0(u)$ solely from the value set of $c_1$.
\end{remark}

\section{Concluding remarks and open problems}\label{sec:open}

The threshold $p>D^{9/2}$ follows from the bound
$t(D)\le D^{3/2}$ and the separation of directions with affine
restrictions. For graphs of actual degree $D$, Corollary~\ref{cor:threshold} gives the more precise
condition $p>\max\{D^4,C(D)\}$. It remains to determine whether a
linear threshold $p>cD$, with an absolute constant $c$, suffices.
In particular, does the weak cylinder conjecture hold for every
graph of total degree $D<p$? Improving the direction count alone
does not remove the condition $p>D^4$ used for exceptionality and
the point count.

A second problem is to determine the largest possible value of
$|\BB_{\rm ar}(f)|$ for a non-cylindrical graph of actual degree $D$.
The estimate $t(D)(D-1)^2+1$ treats the ordered prime factorizations
separately. The power components of degree at least five are already
included by taking zero Dickson parameter. The first degree with two
different orders of distinct prime factors is fifteen, with orders
$3\circ5$ and $5\circ3$. Intersections between the coefficient
conditions for different orders may improve the estimate.

Finally, for fixed $D\ge3$ and $p>D^{9/2}$, put
\[
 M_D(p)=\max\{\,|\NS_S|:\ S=\operatorname{graph}(f),\
             \deg f=D,\ S\text{ is not a cylinder}\,\}.
\]
We have $M_D(p)\le p+1-\lceil p/(D-1)\rceil$, with equality when
$p\equiv1\pmod{D-1}$. Determine $M_D(p)$ for the other residue
classes and classify the cases of equality. By
Corollary~\ref{cor:threshold}, a value larger than $C(D)$ forces
a representation $f(u,s)=c_1(u)s+c_0(u)$; in this case the question
reduces to the size of $c_1(\Fp)$. This connects the extremal problem
to univariate polynomials with small value sets.

\appendix
\section{Homogeneous difference quotients}\label{sec:appendix}

We record a criterion for the leading form of $A_v$ to be a perfect
power. It is independent of the proof of Theorem~\ref{thm:all-degrees}.
Throughout this appendix, $f\in\Fp[x,y]$ has degree $D<p$,
$g=f_D\ne0$, and $\widetilde A_v$ is the homogeneous component of
degree $D-1$ of $A_v$.

\begin{lemma}\label{lem:identity}
For every $v$,
\[
 A_v+t\,\frac{\partial A_v}{\partial t}=(\partial_vf)(P+tv),
 \qquad
 \widetilde A_v+t\,\frac{\partial\widetilde A_v}{\partial t}
   =(\partial_vg)(P+tv),
\]
where $P=(x,y)$. In particular the right-hand side of the second
identity, when nonzero, is a scalar multiple of a product of $D-1$ linear
forms over $\Fpbar$. If $\widetilde A_v\ne0$, every repeated irreducible
factor of $\widetilde A_v$ over $\Fpbar$ is linear.
\end{lemma}

\begin{proof}
By the definition of $A_v$ we have $tA_v=f(P+tv)-f(P)$, and by
Lemma~\ref{lem:degree} we have $t\widetilde A_v=g(P+tv)-g(P)$.
Differentiate the first identity with respect to $t$: on the left
$\partial_t(tA_v)=A_v+t\,\partial_tA_v$, while on the right $f(P)$ does
not involve $t$ and the chain rule gives
\[
 \frac{\partial}{\partial t}f(x+tv_1,y+tv_2)
 =v_1f_x(P+tv)+v_2f_y(P+tv)=(\partial_vf)(P+tv).
\]
The same computation applied to the second identity gives the second
formula.

Assume now $\partial_vg\ne0$. Being a nonzero binary form of degree
$D-1$, it splits over $\Fpbar$ as
$\partial_vg=c\prod_{k=1}^{D-1}L_k$ with $c\ne0$ and $L_k$
linear forms in $x,y$. Each $L_k$ is linear, so
$L_k(P+tv)=L_k(P)+L_k(v)t$, and therefore
\[
 (\partial_vg)(P+tv)=c\prod_{k=1}^{D-1}
   \bigl(L_k(P)+L_k(v)t\bigr)
\]
is a scalar multiple of a product of $D-1$ linear forms of
$\Fpbar[x,y,t]$. This is the second assertion.

For the factor assertion, assume $\widetilde A_v\ne0$; we first check
that $\partial_vg\ne0$. After a linear change of the coordinates $x,y$ we
may assume $v=(0,1)$, so that $\partial_v=\partial_y$: such a change
replaces $g$, $\widetilde A_v$ and $\partial_vg$ by the corresponding
objects, so it affects neither hypothesis nor conclusion. Write
$g=\sum_{i=0}^{D}a_ix^{D-i}y^{i}$. If $g_y=0$, then $ia_i=0$ for
$1\le i\le D$, and $D<p$ gives $a_i=0$ for every $i\ge1$, that is,
$g=a_0x^{D}$; but then $t\widetilde A_v=a_0x^{D}-a_0x^{D}=0$, against the
assumption. Hence $\partial_vg\ne0$, and the previous paragraph applies.

Let finally $F$ be an irreducible factor of $\widetilde A_v$ over
$\Fpbar$ with $F^2\mid\widetilde A_v$, say $\widetilde A_v=F^2H$. Then
$\partial_t\widetilde A_v=F\bigl(2(\partial_tF)H+F\,\partial_tH\bigr)$, so
$F$ divides both $\widetilde A_v$ and $\partial_t\widetilde A_v$, hence
also their combination
$\widetilde A_v+t\,\partial_t\widetilde A_v=(\partial_vg)(P+tv)$. The
latter is a nonzero product of linear forms and $\Fpbar[x,y,t]$ is a
unique factorisation domain, so the irreducible $F$ divides one of those
factors and is therefore linear.
\end{proof}

\begin{theorem}\label{thm:power}
Let $3\le D<p$, let $g$ be the leading form of $f$, let $v\in\PP^1(\Fp)$ and
let $\widetilde A_v$ be the degree-$(D-1)$ homogeneous component of
$A_v$. Then the following are equivalent:
\begin{enumerate}
\item[\textup{(i)}] $\widetilde A_v$ is a perfect $e$-th power for some
$e\ge2$, or $\widetilde A_v\equiv0$;
\item[\textup{(ii)}] $\partial_v^2g\equiv0$;
\item[\textup{(iii)}] $g=\ell^{\,D-1}(\alpha\ell+\beta\ell')$ for linear
forms $\ell,\ell'\in\Fp[x,y]$ with $\ell\ne0$ and $\ell(v)=0$, i.e.\ $g$ has a root of multiplicity
at least $D-1$ and $v$ is its direction.
\end{enumerate}
\end{theorem}

\begin{proof}
(iii)$\Rightarrow$(i). Choose coordinates with $\ell=x$ and $v=(0,1)$, so
that $g=x^{D-1}(\alpha x+\beta y)$ after renaming the coefficients. Then
$t\widetilde A_v=x^{D-1}\bigl(\alpha x+\beta(y+t)\bigr)
 -x^{D-1}(\alpha x+\beta y)=\beta t\,x^{D-1}$, so
$\widetilde A_v=\beta x^{D-1}$, a perfect $(D-1)$-st power (and
$\equiv0$ if $\beta=0$, i.e.\ if $g=\alpha x^{D}$).

(iii)$\Leftrightarrow$(ii). In the same coordinates
$\partial_v^2g=\partial_y^2g$, and $\partial_y^2g=0$ means that $g$ has
degree at most $1$ in $y$, because $D<p$ implies
$k(k-1)\not\equiv0\pmod p$ for $2\le k\le D$. This says exactly
$g=x^{D-1}(\alpha x+\beta y)$.

(i)$\Rightarrow$(ii). If $\widetilde A_v\equiv0$ then $g=c\ell^{D}$ with
$\ell(v)=0$ by Lemma~\ref{lem:degree}(iv), and (ii) holds. So assume
$\widetilde A_v=cU^{e}$ with $e\ge2$ and $U$ a form of degree $(D-1)/e$.
Write $W:=(\partial_vg)(P+tv)$ and note $W\ne0$: otherwise
$\partial_vg\equiv0$, which forces $g=c\ell_v^{\,D}$ and hence
$\widetilde A_v\equiv0$ by Lemma~\ref{lem:degree}(iv). By
Lemma~\ref{lem:identity},
\[
 c\,U^{e-1}\bigl(U+e\,t\,\partial_tU\bigr)=W .
\]
Since $W$ is a product of linear forms and $e-1\ge1$, every irreducible
factor of $U$ is one of those linear forms; writing
$\partial_vg=\prod_jL_j^{\nu_j}$ with the $L_j$ pairwise
non-proportional and $\widehat L_j:=L_j(P)+L_j(v)t$
{(these are again pairwise non-proportional as linear forms in
$x,y,t$, since their restrictions to $t=0$ are)}, we get
$U=\gamma\prod_j\widehat L_j^{m_j}$. Restricting to $t=0$ and using
$\widetilde A_v|_{t=0}=\partial_vg$ (Lemma~\ref{lem:degree}(i)) gives
$c\gamma^{e}\prod_jL_j^{em_j}=\prod_jL_j^{\nu_j}$, whence
$\nu_j=em_j$ for all $j$ and therefore
\[
 \widetilde A_v=c\,U^{e}=\kappa\,W\qquad\text{for some }\kappa\ne0 .
\]
Substituting this into the identity of Lemma~\ref{lem:identity} and using
$\partial_tW=(\partial_v^2g)(P+tv)$ yields
\[
 \kappa\,t\,(\partial_v^{2}g)(P+tv)=(1-\kappa)\,(\partial_vg)(P+tv).
\]
As $(x,y,t)\mapsto(P+tv,t)$ is an invertible linear change of variables,
this is an identity of polynomials in $(Q,t)$ with $Q$ and $t$
independent: $\kappa\,t\,\partial_v^{2}g(Q)=(1-\kappa)\partial_vg(Q)$.
{The right-hand side does not involve $t$, while the left-hand side
is divisible by $t$; hence both sides vanish. Since $\partial_vg\ne0$,
the vanishing of the right-hand side gives $\kappa=1$, and then the
vanishing of the left-hand side gives $t\,\partial_v^2g(Q)\equiv0$, that
is, $\partial_v^2g\equiv0$.}
\end{proof}

\begin{remark}\label{rem:degree2power}
The restriction $D\ge3$ is necessary. For $g=xy$ and $v=(1,0)$ one has
$\partial_v^2g=0$ but $\widetilde A_v=y$, which is not a perfect $e$-th
power with $e\ge2$; moreover both coordinate directions satisfy
$\partial_v^2g=0$.
\end{remark}

\section*{Acknowledgements}

Both authors are members of the Gruppo Nazionale per le Strutture
Algebriche, Geometriche e le loro Applicazioni (GNSAGA) of the Istituto
Nazionale di Alta Matematica (INdAM), whose support is gratefully
acknowledged.

\end{document}